\documentclass[reqno]{amsart}     

\usepackage{amssymb, amsmath, amsthm, epsf, epsfig}

\renewcommand{\(}{\left(}
\renewcommand{\)}{\right)}
\newtheorem{theo}{Theorem}
\newtheorem{prop}{Proposition}
\newtheorem{lemma}{Lemma}
\newtheorem{cor}{Corollary\!\!}

\newtheorem{ncor}{Corollary}

\theoremstyle{definition}

\newtheorem{df}{Definition}
\newtheorem{ex}{Example}

\theoremstyle{remark}

\newtheorem{rem}{Remark\!\!}
 
\newtheorem{nrem}{Remark}

\newcommand{\beq}{\begin{equation}} 
\newcommand{\eeq}{\end{equation}} 
\newcommand{\bal}{\begin{align}} 
\newcommand{\eal}{\end{align}} 
\newcommand{\bals}{\begin{align*}} 
\newcommand{\eals}{\end{align*}} 
\newcommand{\barr}[1]{\begin{array}{#1}} 
\newcommand{\earr}{\end{array}}

\newcommand{\bth}{\begin{theo}} 
\newcommand{\bl}{\begin{lemma}} 
\newcommand{\el}{\end{lemma}} 
\newcommand{\bp}{\begin{prop}} 
\newcommand{\ep}{\end{prop}} 
\newcommand{\bdf}{\begin{df}} 
\newcommand{\edf}{\end{df}} 
\newcommand{\brem}{\begin{rem}} 
\newcommand{\erem}{\end{rem}} 
\newcommand{\bnrem}{\begin{nrem}} 
\newcommand{\enrem}{\end{nrem}} 
\newcommand{\bex}{\begin{ex}} 
\newcommand{\eex}{\end{ex}} 
\newcommand{\bcor}{\begin{cor}} 
\newcommand{\ecor}{\end{cor}} 
\newcommand{\bncor}{\begin{ncor}} 
\newcommand{\encor}{\end{ncor}} 
\newcommand{\bpf}{\begin{proof}} 
\newcommand{\epf}{\end{proof}}

\def\({\left(} 
\def\){\right)}

\numberwithin{equation}{section}

\title{Universal singular exponents in catalytic variable equations}

\author{Michael Drmota$^*$, Marc Noy\textsuperscript{\dag{}} {}, and Guan-Ru Yu$^*$}

\thanks{{}$^*$ TU Wien, Institute of Discrete Mathematics and Geometry,
Wiedner Hauptstrasse 8-10, A-1040 Vienna, Austria. michael.drmota@tuwien.ac.at. Research 
supported by the
Austrian  Science Foundation FWF, project F 50-02.}

\thanks{\textsuperscript{\dag{}} Universitat Polit\`ecnica de Catalunya, 
Department of Mathematics, Pau Gargallo 14, 08028 Barcelona, Spain. marc.noy@upc.edu. Research supported in part by Ministerio de Ciencia, Innovaci\'on y Univerisades, grant MTM2017-82166-P.}

\begin{document}

\begin{abstract}
Catalytic equations appear in several combinatorial applications, most notably in the numeration of lattice path and in the enumeration  of planar maps. The main purpose of this 
paper is to show that the asymptotic estimate for  the coefficients of the solutions of (so-called) positive
catalytic equations has a universal asymptotic behavior. In particular,  this provides 
a rationale why the number of maps of size $n$ in various planar map classes  
grows asymptotically like $c\cdot  n^{-5/2} \gamma^n$, 
for suitable positive constants $c$ and $\gamma$. 
Essentially we have to distinguish between linear catalytic equations (where the subexponential growth
is $n^{-3/2}$) and non-linear catalytic equations (where we have $n^{-5/2}$ as in planar maps).
Furthermore we provide a quite general central limit theorem for parameters that can 
be encoded by catalytic functional equations, even when they are not positive.
\end{abstract}

\maketitle

\section{Introduction}
\label{sec1}

A planar map is a connected planar graph, possibly with loops and multiple edges,
together with an embedding in the plane. A map is rooted if a vertex $v$ and an edge
$e$ incident with $v$ are distinguished, and are called the root-vertex and root-edge,
respectively. The face to the right of e is called the root-face and is usually taken
as the outer face. All maps in this paper are rooted.

The enumeration of rooted maps is a classical subject, initiated by Tutte in the
1960's. Tutte (and Brown) introduced the technique now called ``the quadratic method'' in order to
compute the number $M_n$ of rooted maps with $n$ edges, proving the formula
\[
M_n = \frac{2(2n)!}{(n+2)!n!}3^n.
\]
This was later extended by Tutte and his school to several classes of planar maps:
2-connected, 3-connected, bipartite, Eulerian, triangulations, quadrangulations, etc.
Using the previous formula, Stirling's estimate gives $M_n \sim  (2/\sqrt \pi) \cdot n^{-5/2} 12^n$. 
In all cases where a ``natural'' condition 
is imposed on maps,
the asymptotic estimates turn out to be of this kind:
\begin{equation}\label{1.1}
c\cdot  n^{-5/2} \gamma^n.
\end{equation}
The constants $c$ and $\gamma$ depend on the class 
under consideration, but one gets systematically an $n^{-5/2}$ term in the estimate. 

This phenomenon is discussed by Banderier
et al. \cite{1}: `This
generic asymptotic form is ``universal'' in so far as it is valid for all known ``natural
families of maps''.'  
The goal of this paper is to provide an explanation for this universal phenomenon, based on a detailed analysis 
of  functional equations for  generating functions with a catalytic variable.

In order to motivate the
statements that follow, let us recall the basic technique for counting planar maps.
Let $M_{n,k}$ be the number of maps with $n$ edges and in which the degree of the root-face is equal $k$. Let $M(z, u) = \sum_{n,k} M_{n,k}u^kz^n$ 
be the associated generating function.
As shown by Tutte \cite{16}, $M(z,u)$ satisfies the quadratic equation
\begin{equation}\label{1.3}
M(z, u) = 1 + zu^2M(z, u)^2 + uz
\frac{ uM(z, u) -M(z, 1)}{u-1}.
\end{equation}
The variable $u$ is called a ``catalytic variable''. 

It turns out that 
\begin{equation}\label{eqMz1}
M(z,1) = \sum_{n\ge 0} M_n z^n = \frac{18z - 1 + (1 - 12z)^{3/2}}{54z^2}
 = 1 + 2z + 9z^2 + 54z^3 + \cdots ,
\end{equation}
from which we can deduce the explicit form for the numbers $M_n$. 
The remarkable observation here is the singular part $(1 - 12z)^{3/2}$
that reflects the asymptotic behavior $c\cdot n^{-5/2} 12^n$ of $M_n$.

A general approach to equations of the form (\ref{1.3}) 
was carried out by Bousquet-M\'elou and Jehanne \cite{BM-J}. First one
rewrites (\ref{1.3}) into the form
\begin{equation}\label{eqBM}
P(M(z,u),M_1(z),z,u) = 0,
\end{equation}
where $P(x_0,x_1,z,u)$ is a polynomial (or more generally an analytic function)
and $M_1(z)$ abbreviates $M(z,1)$.
Next one searches for functions  $f(z)$, $y(z)$ and  $u(z)$ with
\begin{align*}
P(f(z),y(z),z,u(z)) &= 0,\\
P_{x_0}(f(z),y(z),z,u(z)) &= 0,\\
P_u(f(z),y(z),z,u(z)) &= 0.
\end{align*}
If $y(z)$ has a power series representation at $z=0$ then one has $M_1(z) = y(z)$.

Bousquet-M\'elou and Jehanne \cite{BM-J} considered in particular equations of the form\footnote
{Actually Bousquet-M\'elou and Jehanne \cite{BM-J} considered more general functional equations
that contain also higher differences.}
\begin{equation}\label{eqBMJ}
\boxed{M(z,u) = F_0(u) + z Q\left( M(z,u), \frac{M(z,u)-M(z,0)}u, z,u \right)},
\end{equation}
where $F_0(u)$ and $Q(\alpha_0,\alpha_1,z,u)$ are polynomials, and showed that there is a unique power series solution 
$M(z,u)$ that is also an algebraic function. Actually all the examples that we will discuss can be rewritten
into (almost) this form (possibly by replacing $u$ by $u+1$). For example, for the equation (\ref{1.3})
we have 
\begin{equation}\label{fqu}
F_0(u) = 1 \quad \mbox{and} \quad Q(\alpha_0,\alpha_1,z,u) = (u+1)^2\alpha_0^2 + (u+1)\alpha_0 + (u+1)\alpha_1.
\end{equation}

In the context of this paper we  always assume that $F_0$ and $Q$ have non-negative coefficients
This is natural since  Equation (\ref{eqBMJ}) can be seen as a translation of a recursive
combinatorial description of maps or other combinatorial objects.  
This also implies that $M(z,u)$ has non-negative coefficients, since the equaytion (\ref{eqBMJ}) can
be written as an infinite system of equation for the functions $M_j(z) = [u^j]\, M(z,u)$ with non-negative
coefficients on the right hand side.

Let us consider the first case, where $Q$ is linear in $\alpha_0$ and $\alpha_1$, that is, we can write  (\ref {eqBMJ}) as 
\begin{equation}\label{eqlinear}
\boxed{M(z,u) = Q_0(z,u) + z M(z,u) Q_1(z,u) +  z\frac{M(z,u)-M(z,0)}u  Q_2(z,u)}.
\end{equation}
Here we are in the framework of the so-called {\it kernel method}.
We rewrite (\ref{eqlinear}) as
\begin{equation}\label{eqkerneleq}
M(z,u)( u - zuQ_1(z,u) - z Q_2(z,u)) = uQ_0(z,u) - z M(z,0) Q_2(z,u),
\end{equation}
where
\[
K(z,u) = u - zuQ_1(z,u) - z Q_2(z,u)
\]
is the \emph{kernel}. The idea of the kernel method is to bind $u$ and $z$ so that
$K(z,u)=0$, that is, one considers a function $u=u(z)$ such that $K(z,u(z)) =0$.
Then the left hand side of (\ref{eqkerneleq}) cancels and $M(z,0)$ can be calculated
from the right hand side by setting $u=u(z)$.

\begin{prop}\label{ProTh1}
Suppose that $Q_0$, $Q_1$, and $Q_2$ are polynomials in $z$ und $u$ with non-negative
coefficients and let $M(z,u)$ be the power series solution of (\ref{eqlinear}). 
Furthermore let $u(z)$ be the power series solution of the equation
\[
u(z) = zQ_2(z,u(z)) +  z u(z) Q_1(z,u(z)), \quad \mbox{with $u(0) = 0$.}
\]
Then $M(z,0)$ is given by
\[
M(z,0) = \frac{Q_0(z,u(z))}{1-zQ_1(z,u(z))}.
\]
\end{prop}

There are three particular {\it degenerate} cases, where the solution function $M(z,0)$ is a 
a rational function (or even a polynomial). In these cases  the asymptotic analysis of $M_n$ is trivial:

\begin{itemize}
\item If $Q_0= R_0(z)$ and $Q_1= R_1(z)$  depend only on $z$ then 
\[
M(z,0) = \frac{R_0(z)}{1-zR_1(z)}.
\]
\item If $Q_1= Q_{10}(z)$  depends only on $z$ and if $Q_2 = T_0(z) + T_1(z) u$ is at most linear in $u$ then
\[
u(z) = \frac{zT_0(z)}{1 - zR_1(z) - zT_1(z)}\quad \mbox{and}\quad
M(z,0) = \frac{Q_{0}(z,u(z))}{1-zR_1(z)} 
\]
are rational functions.
\item
If $Q_2$ has $u$ as a factor then $u(z) = 0$ and we have
\[
M(z,0) = \frac{Q_0(z,0)}{1-zQ_1(z,0)}
\]
is a rational function.
\end{itemize}

In all other cases $M(z,0)$ has universally 
a dominant square root singularity.

\begin{theo}\label{Th1}
Suppose that $Q_0$, $Q_1$, and $Q_2$ are polynomials in $z$ und $u$ with non-negative coefficients such that none of the three above mentioned cases occurs.

Let $M(z,u)$ be the power series solution of (\ref{eqlinear}) and let $z_0>0$ denote
the radius of convergence of $M(z,0)$. Then the local Puiseux expansion of $M(z,0)$ around $z_0$ is 
given by 
\begin{equation}\label{eqsqrtsing}
M(z,0) = a_0+a_1(1-z/z_0)^{1/2}+a_2(1-z/z_0)+\cdots,
\end{equation}
where $a_0 > 0$ and $a_1 < 0$. Furthermore there exists $b\ge 1$, a non-empty set $J \subseteq \{0,1,\ldots,b-1\}$ 
of residue classes modulo $b$ and constantcs $c_j> 0$ for such that for $j\in J$
\begin{equation}\label{eqTh1.2}
M_n = [z^n]M(z,0) = c_j n^{-3/2} z_0^{-n}\left( 1 + O\left( \frac 1n \right) \right), \qquad (n\equiv j \bmod b,\ n\to \infty)
\end{equation}
and $M_n = 0$ for $n\equiv j \bmod b$ with $j\not \in J$.
\end{theo}
This result is quite easy to prove (see Section \ref{pth1}). We just want to mention that
there are variations of the above model, for example equations of the form 
\[
M(z,u) = Q_0(z,u) + z M(z,u) Q_1(z,u) +  z\frac{M(z,u)-M(z,0)}u  Q_2(z,u) + u M(z,0) Q_3(z,u),
\]
that can be handled in the same way; see \cite{Prod}. However, the asymptotics can be slightly
different. For example one might have $n^{-1/2}$ instead of $n^{-3/2}$ in the subexponential growth of $M_n$
(namely if $z_0Q_1(z_0,u(z_0)) + Q_3(z_0,u(z_0)) = 1$; if $Q_3=0$ then we y have
$z_0Q_1(z_0,u(z_0)) < 1$).

In the non-linear case the situation is more involved. Here we find the 
solution function $M(z,0)$ in the following way.

\begin{prop}\label{ProTh2}
Suppose that $Q$ is a polynomial in $\alpha_0,\alpha_1,z,u$ with non-negative coefficients
that depends (at least) on $\alpha_1$, that is, $Q_{\alpha_1} \ne 0$, and 
let $M(z,u)$ be the power series solution of (\ref{eqBMJ}).
Furthermore we assume that $Q$ is not linear in $\alpha_0$ and $\alpha_1$, that is,
$Q_{\alpha_0\alpha_0} \ne 0$, or $Q_{\alpha_0\alpha_1} \ne 0$ or $Q_{\alpha_1\alpha_1} \ne 0$.

Let $f(z), u(z), w(z)$  be the power series solution of the system of equations 
\begin{align}
f(z) &=  F_0(u(z)) + zQ(f(z), w(z), z,u(z)), \nonumber \\
u(z) &=  z u(z)  Q_{\alpha_0} (f(z), w(z), z,u(z)) + z  Q_{\alpha_1} (f(z), w(z), z,u(z)),  \label{eqnewsystem} \\
w(z)  &= F_{0}'(u(z)) + z Q_u (f(z), w(z), z,u(z)) + z w(z)   Q_{\alpha_0} (f(z), w(z), z,u(z)).\nonumber 
\end{align}
with $f(0) = F_0(0)$, $u(0) = 0$, $w(0) = F_{0}'(0)$. Then
\[
M(z,0) = f(z) - w(z)u(z).
\]
\end{prop}
\noindent
The meaning of $w(z)$ will become clear later in the proof of the proposition in Section 4. 
In Theorem~\ref{Th2} we  assume that $Q_{\alpha_0u} \ne 0$, which implies that the system (\ref{eqnewsystem})
is strongly connected. This means that the  dependency di-graph of the system is strongly connected as 
discussed in Section~\ref{pth2}.

Again there are some {\it degenerate} cases. We do not give a complete list and we just  discuss  some of them. 
We also comment on the case $Q_{\alpha_1} = 0$. 
Given a multivariate function $f$ we replace one of its variables with a dot if $f$ actually does not depend on this variable. 
\begin{itemize}
\item Suppose that $Q_u = F_{0}' = 0$, that is, $F_0$ is constant and  $Q$ does not depend on $u$. 
Here $w(z) = 0$ and consequently
\[
M(z,0) = f(z),
\]
where $f(z)$ is the solution of the equation
\[
f(z) = F_0 + zQ(f(z),0,z,\cdot).
\]
Thus, depending on the degree of $\alpha_0$ in $Q(\alpha_0,0,z,\cdot)$, the solution function $M(z,0)$ 
is either a polynomial, a rational function, or  
it has a square-root singularity as in (\ref{eqsqrtsing}); see \cite{0}.
\item Next suppose that $Q_u = Q_{\alpha_0} = 0$ but $F_{0}'\ne 0$.
Here we have we are left with the equations 
\[
f = F_0(u) + zQ(w,z), \quad u = zQ_{\alpha_1}(w,z),\quad w = F_{0}'(u).
\]
Thus, we have to solve the equation $u = zQ_{\alpha_1}(F_{0}'(u),z)$ to obtain $u=u(z)$ and 
consequently $w(z) = F_{0}'(u(z))$ and $f(z) = F_0(u(z)) + zQ(w(z),z)$.
Hence, depending on the structure of $zQ_{\alpha_1}(F_{0}'(u),z)$ we obtain a polynomial, a rational function, or 
a square-root singularity for 
\begin{align*}
M(z,0) &= f(z) - w(z)u(z) \\
& = z Q(w(z),z)+ F_0 ( u(z) ) - u(z) F_{0}'( u(z) ).
\end{align*}
\item
Finally, if $Q_{\alpha_1} = 0$ then 
we have an equation of the form
\[
M(z,u) = F_0(u) + zQ(M(z,u),z,u).
\]
In this case the catalytic variable $u$ is not necessary and we can set it to $0$.
Hence, depending on the structure of $Q$ we just get a polynomial, a rational function,
or a square-root singularity for $M(z,0)$ (see \cite{0}).
\end{itemize}

\begin{theo}\label{Th2}
Suppose that $Q$ is a polynomial in $\alpha_0,\alpha_1,z,u$ with non-negative coefficients
that depends (at least) on $\alpha_1$, that is, $Q_{\alpha_1} \ne 0$ and 
let $M(z,u)$ be the power series solution of (\ref {eqBMJ}).
Furthermore we assume that $Q$ is not linear in $\alpha_0$ and $\alpha_1$, that is,
$Q_{\alpha_0\alpha_0} \ne 0$ or $Q_{\alpha_0\alpha_1} \ne 0$ or $Q_{\alpha_1\alpha_1} \ne 0$.
We assume additionally that $Q_{\alpha_0u} \ne 0$.

Let $z_0>0$ denote the radius of convergence of $M(z,0)$. 
Then the local Puiseux expansion of $M(z,0)$ around $z_0$ is 
given by 
\begin{equation}\label{eqTh2.1}
M(z,0) = a_0+a_2(1-z/z_0)+a_3(1-z/z_0)^{3/2}+O((1-z/z_0)^2),
\end{equation}
where $a_0>0$. 

If we further assume that $F_0'(0) = 0$ and $Q_{\alpha_1}(F_0(0),0,0,0) \ne 0$, 
then $a_3 > 0$. In this case there exists $b\ge 1$ and a residue class $a$ modulo $b$ such that
\begin{equation}\label{eqTh2.2}
M_n = [z^n]M(z,0)  c\, n^{-5/2} z_0^{-n} \left( 1 + O\left( \frac 1n \right) \right), \qquad (n\equiv a \bmod b,\ n\to \infty)
\end{equation}
for some constant $c> 0$, and $M_n = 0$ for $n\not \equiv  a \bmod b$.
\end{theo}

The plan of the paper is as follows. In the next section we collect some 
application examples of Theorems~\ref{Th1} and \ref{Th2}.
We then prove Proposition~\ref{ProTh1} and Theorem~\ref{Th1} in Section~\ref{pth1},
and Proposition~\ref{ProTh2} and Theorem~\ref{Th2} in Section~\ref{pth2}.
Finally we provide more information on the solution of catalytic equations.
In particular we formulate a quite general central limit theorem, involving an additional parameter, in  Section~\ref{sec:clt}.

\section{Examples}\label{sec:ex}

\subsection{The linear case}

Natural examples for the linear case (Proposition~\ref{ProTh1} and Theorem~\ref{Th1}) come from the enumeration of 
lattice path. We consider paths starting from the coordinate point $(0,0)$ (or from $(0,t)$, $t\in \mathbb{N}$) 
and allowed to move only to the right (up, straight or down), but forbid going below the $x$-axis $y=0$ at each step.
Define a step set $\mathcal{S}=\{(a_1,b_1),(a_2,b_2),\cdots,(a_s,b_s)|(a_j,b_j)\in \mathbb{N}\times\mathbb{Z}\}$,
and let $f_{n,k}$ be the number of paths ending at point $(n,k)$, where each step is in $\mathcal{S}$. The associated generating function is then defined as
  \[
F(z,u)=\sum_{n,k\ge 0}{f_{n,k}z^nu^k}.
\]
\begin{ex}(Motzkin Paths) We start from $(0,0)$ with step set $\mathcal{S} = \{(1,1),(1,0),(1,-1)\}$. The functional equation of its associated generating function is as follows:
\begin{align*}
F(z,u)&=1+z\left(u+1+\frac{1}{u}\right)F(z,u)-\frac{z}{u}F(z,0)\\
&=1+z(u+1)F(z,u)+z\frac{F(z,u)-F(z,0)}{u},
\end{align*}
which in the notation of   \eqref{eqlinear} corresponds to 
\[
Q_0(z,u)=1,\quad Q_1(z,u)=u+1,\quad \text{and} \quad Q_2(z,u)=1.
\]
We let $u(z)$ be the power series solution of the equation
\[
u(z) = zQ_2(z,u(z)) + zu(z)Q_1(z,u(z)) = z+z u(z)(1+u(z)),
\]
that is,
\[
u(z) = \frac{1-z-\sqrt{1-2z-3z^2}}{2z}.
\]
Then $F(z,0)$ is given by
\[
F(z,0) = \frac{Q_0(z,u(z))}{1-zQ_1(z,u(z))}=\frac{1}{1-z(1+u(z))}=\frac{1-z-\sqrt{1-2z-3z^2}}{2z^2},
\]
and
\begin{align*}
M^*_n=f_{n,0}=[z^n]F(z,0) = \sum_{k=0}^{\lfloor n/2\rfloor}{\frac{n!}{(n-2k)!k!(k+1)!}}
\sim \frac{3\sqrt 3}{2\sqrt \pi} n^{-3/2} 3^n.
\end{align*}
These numbers are also called ``Motzkin numbers''.
\end{ex}
\begin{ex}
We start from $(0,k_0)$ with step set $\mathcal{S} = \{(2,0),(1,-1)\}$. 
Here the functional equation  is given by
\begin{align*}
F(z,u)&=u^{k_0}+(z^2+\frac{z}{u})F(z,u)-\frac{z}{u}F(z,0)\\
&=u^{k_0}+z^2F(z,u)+z\frac{F(z,u)-F(z,0)}{u},
\end{align*}
which corresponds to 
\[
Q_0(z,u)=u^{k_0},\quad Q_1(z,u)=z\quad \text{and} \quad Q_2(z,u)=1.
\]
This is actually a {\it degenerate} case since $Q_1$ and $Q_2$  depend only on $z$. 
Here $u(z)$ is a rational function
\[
u(z) = \frac{zQ_2(z,\cdot)}{1-zQ_1(z,\cdot)} = \frac{z}{1-z^2},
\]
as well as 
\[
F(z,0) = \frac{Q_0(z,u(z))}{1-zQ_1(z,u(z))}=\frac{u(z)^{k_0}}{1-z^2}=\frac{z^{k_0}}{(1-z^2)^{k_0+1}}.
\]
\end{ex}
\begin{ex}
We start again from $(0,0)$ but now with step set $\mathcal{S}= \{(2,0),(1,1),(1,0)\}$, and we also assume that the step $(1,0)$ 
is forbidden on the $x$-axis $y = 0$. The functional equation in this case is 
\begin{align*}
F(z,u)&=1+z(z+u+1)F(z,u)-zF(z,0)\\
&=1+z(z+u)F(z,u)+zu\frac{F(z,u)-F(z,0)}{u},
\end{align*}
that is, we have
\[
Q_0(z,u)=1,\quad Q_1(z,u)=z+u\quad \text{and} \quad Q_2(z,u)=u.
\]
Here $Q_2$ has $u$ as a factor so that we are again in a {\it degenerate} case.
We have $u(z) = 0$ and consequently
\[
F(z,0)=\frac{1}{1-z^2}=1+z^2+z^4+z^6+\cdots.
\]
\end{ex}

\subsection{The non-linear case}

We collect here some examples from the enumeration of planar maps. The starting point is 
the classical example of all planar maps.

\begin{ex}
Let $M(z, u)$ be the generating function  of planar maps with $n$ edges and in which the degree of the root-face is equal $k$.
We have already mentioned that $M(z,u)$ satisfies the non-linear catalytic equation (\ref{1.3}). 
In order to apply Proposition~\ref{ProTh2} and  Theorem~\ref{Th2} we use the substitution $u\to u+1$ and obtain
\[
M(z, u+1)=1+z(u+1)\left( (u+1)M(z, u+1)^2 + M(z, u) + \frac{M(z, u+1)-M(z, 1+0)}{u}\right),
\]
that is, we have $F_0(u) = 1$, and $Q(\alpha_0,\alpha_1,z,u)= (u+1)^2\alpha_0^2 + (u+1)\alpha_0 + (u+1)\alpha_1$.
Here $Q_{\alpha_1} = u+1 \ne 0$, $Q_{\alpha_0,u} \ne 0$, and $Q_{\alpha_0,\alpha_0} \ne 0 $, so that 
Theorem~\ref{Th2} applies. Of course this is in accordance with 
\[
M(z,1) = \sum_{n\ge 0} M_n z^n = \frac{18z - 1 + (1 - 12z)^{3/2}}{54z^2},
\]
and 
\[
M_n=[z^n]M(z,1)\sim \frac{2}{\sqrt{\pi}} n^{-5/2} 12^{n}.
\]
\end{ex}

\begin{ex}
Let $E(z,u)$ be the  generating function of bipartite planar maps
which satisfies the catalytic equation
\[
E(z, u) = 1 + zu^2E(z, u)^2 + u^2z
\frac{ E(z, u) -E(z, 1)}{u^2-1}.
\]
Here we use the substitution $u = \sqrt{1 +v}$ and obtain
\[
 E(z, \sqrt{1 +v}) = 1 + z(v+1) E(z, \sqrt{1 +v})^2 + (v+1)z
\frac{ E(z, \sqrt{1 +v}) - E(z, 1)}{v},
\]
which is of a type where Theorem~\ref{Th2} applies:
\[
F(v) = 1, \quad Q(\alpha_0,\alpha_1,z,v) = \alpha_0^2(v+1) + \alpha_1(v+1).
\]
\end{ex}

\begin{ex}
Next let $B(z,u)$ be the generating function of $2$-connected planar maps. It satisfies
\[
B(z, u) = z^2u + zu B(z,u) + u(z+B(z,u)) \frac{ B(z, u) -B(z, 1)}{u-1}.
\]
After substituting $u$ by $u+1$ we obtain
\[
B(z, u+1) = z^2(u+1) + z(u+1) B(z,u+1) + (u+1)(z+B(z,u+1)) \frac{ B(z, u+1) - B(z, 1)}{u},
\]
which is not exactly of the form (\ref{eqBMJ}). Nevertheless   
the same methods as in the proof of Theorem~\ref{Th2} apply -- we just have to observe that
the analogue of the system of equations (\ref{eqnewsystem.2}) has proper positive
power series solutions -- and we obtain the same result. 
\end{ex}

\begin{ex}
Finally let $T(z,u)$ the generating function for planar triangulations, which  satisfies
\[
T(z, u) = (1-uT(z,u)) + (z+u)T(z,u)^2 + z(1-uT(z,u))\frac{ T(z, u) -T(z, 0)}{u} .
\]
In order to get rid of the negative sign we use the substitution $\widetilde T(z,u) = T(z,u)/(1-uT(z,u))$ and we obtain 
\[
\widetilde T(z, u) = 1 + u \widetilde T(z,u) + z(1 + \widetilde T(z,u))\frac{ \widetilde T(z, u) -\widetilde T(z, 0)}{u}.
\]
Again this is not precisely of the form (\ref{eqBMJ}) but our methods apply once more. 
Note that $\widetilde T(z, 0) = T(z, 0)$.
\end{ex}

\section{Proofs of Proposition~\ref{ProTh1} and Theorem~\ref{Th1}}\label{pth1}

\subsection{Proof of Proposition~\ref{ProTh1}}

As already mentioned in the introduction, we rewrite (\ref{eqlinear}) as
\[
M(z,u) \left(u - zu Q_1(z,u) -  z Q_2(z,u) \right) = uQ_0(z,u) - z M(z,0) Q_2(z,u).
\]
It is clear that if $u=u(z)$ satisfies 
\begin{equation}\label{equequ}
u = zQ_2(z,u) + z u Q_1(z,u),
\end{equation}
then the {\it kernel}
$K(z,u) = u - zu Q_1(z,u) - z Q_2(z,u)$ is identically zero, which implies that
$M(z,0)$ is given by $M(z,0) = u(z) Q_0(z,u(z))/(z Q_2(z,u(z)))$. 
Since $z Q_2(z,u(z)) = u(z) (1-z Q_1(z,u(z)))$, we also have 
$$M(z,0) =  \frac{Q_0(z,u(z))}{1-z Q_1(z,u(z))},$$ as claimed.

Finally we mention that  Equation (\ref{equequ}) has always a unique power series solution
$u = u(z)$ with $u(0)= 0$. On a formal level this is immediately clear by comparing coefficients and
rewriting (\ref{equequ}) as a recurrence for the coefficients of $u(z)$. 
However, (\ref{equequ}) can be also seen as a fixed point equation, which is a contraction 
if $z$ and $u$ are sufficiently small. This means that the recurrence $u_0(z) = 0$, 
$u_{k+1}(z) =  zQ_2(z,u_k(z)) +z u_k(z) Q_1(z,u_k(z))$, $k\ge 0$, has an analytic limit
$u(z)$, provided that $z$ is sufficiently small in modulus.

\subsection{Proof of  Theorem~\ref{Th1}}

Since $u(z)$ has non-negative 
coefficients the dominant singularity is positive and equals the radius of convergence
of $u(z)$. 

We do not comment on the {\it degenerate cases} that are discussed after  Theorem~\ref{Th1}, since there are only rational functions. In the non-degenerate case  Equation (\ref{equequ})
is a non-linear positive polynomial equation for $u(z)$. Here it  follows by general considerations that
$u(z)$ has a square-root singularity at the radius of convergence~$z_0$
\[
u(z) = u_0 + u_1 (1-z/z_0)^{1/2} + u_2 (1-z/z_0) + u_3 (1-z/z_0)^{3/2} + \cdots,
\]
where $z_0>0$ and $u_0 >0$ are (uniquely) given by the system of equations
\begin{align*}
u_0 &= z_0Q_2(z_0,u_0) + z_0 u_0 Q_1(z_0,u_0), \\
1 &= z_0Q_{2,u}(z_0,u_0) + z_0 Q_1(z_0,u_0) + z_0 u_0 Q_{1,u}(z_0,u_0).
\end{align*} 
See \cite{0} and \cite{Drm-randomtrees} for details. In particular we have $u_1 < 0$.

More precisely, $u(z)$ can be represented as $u(z) = z^a U(z^b)$, where $a\ge 0$,
$b\ge 1$, and $U(z)$ has also
a square-root singularity at $z = z_0^{1/b}$, that is the only singularity on the circle
$|z|\le z_0^{1/b}$. In particular it follows that the coefficients of $U_k = [z^k] U(z)$ are
asymptotically given by $U_k \sim c k^{-3/2} z_0^{-k/b}$ for some $c> 0$. 
With the help of $U(z)$ we can now completely describe $M(z,0)$ and directly
deduce (\ref{eqTh1.2}).

\section{Proof of Proposition~\ref{ProTh2} and Theorem~\ref{Th2}}\label{pth2}

\subsection{Proof of Proposition~\ref{ProTh2}}

As mentioned in the introduction, general catalytic equations can 
be solved with the help of the method of Bousquet-M\'elou and Jehanne~\cite{BM-J}. 
For this purpose we set	
\begin{equation}\label{eqPrep}
P(x_0,x_1,z,u) = F_0(u) + zQ(x_0, (x_0-x_1)/u, z,u) - x_0.
\end{equation}
The next step is to find functions $x_0 = f(z)$, $x_1 = y(z)$, and $u = u(z)$ such that
$P = 0$, $P_{x_0} = 0$, and $P_u = 0$. In our situation this means that
\begin{align*}
f(z) &=  F_0(u(z)) + zQ(f(z), (f(z)-y(z))/u(z), z,u(z)),\\
1 &=  z Q_{\alpha_0} (f(z), (f(z)-y(z))/u(z), z,u(z)) \\
&+ \frac z{u(z)}  Q_{\alpha_1} (f(z), (f(z)-y(z))/u(z), z,u(z)),\\
0  &= F_{0}'(u(z)) + z Q_u (f(z), (f(z)-y(z))/u(z), z,u(z)) \\& - 
z \frac{f(z)-y(z)}{u(z)^2}  Q_{\alpha_1} (f(z), (f(z)-y(z))/u(z), z,u(z)).
\end{align*}
In order to simplify this system we set $w=w(z) = (f(z)-y(z))/u(z)$, multiply the second equation by $u(z)$
and replace $z Q_{\alpha_1}/u(z)$ by $1- z Q_{\alpha_0}$ in the third equation. This leads to the system
\begin{align}
f(z) &=  F_0(u(z)) + zQ(f(z), w(z), z,u(z)), \nonumber \\
u(z) &=  z u(z)  Q_{\alpha_0} (f(z), w(z), z,u(z)) + z  Q_{\alpha_1} (f(z), w(z), z,u(z)),  \label{eqnewsystem.2} \\
w(z)  &= F_{0}'(u(z)) + z Q_u (f(z), w(z), z,u(z)) + z w(z)   Q_{\alpha_0} (f(z), w(z), z,u(z)),\nonumber 
\end{align}
which is precisely \eqref{eqnewsystem}.
This is a (so-called) positive polynomial system of equations for the unknown functions
$f(z)$, $w(z)$, and $u(z)$; recall that the coefficients of $F_0$ and $Q$ are non-negative.
It is easy to show that the system (\ref{eqnewsystem.2}) has unique power series solutions
with $f(0) = F_0(0)$, $w(0) = F_{0}'(0)$, $u(0) = 0$ and non-negative coefficients.
Thus, $y(z) = f(z)-u(z) w(z)$ is a power series, too, and consequently 
$M(z,0) = y(z) = f(z)-u(z) w(z)$.
 
\subsection{Proof of  Theorem~\ref{Th2}}

Positive polynomial systems of equations are discussed in detail in \cite{0}. In particular if the system is 
strongly connected then we know that there is a common dominant singularity $z_0$
and $f(z)$, $w(z)$, and $u(z)$ have a square root singularity at $z_0$ of the form
(\ref{eqsqrtsing}):
\begin{align}
f(z) &= f_0 + f_1 Z + f_2Z^2 + f_3 Z^3 + \cdots,   \nonumber \\
u(z) &= u_0 + u_1 Z + u_2Z^2 + u_3 Z^3 + \cdots,   \label{eqsingrep-3} \\
w(z)  &= w_0 + w_1 Z + w_2Z^2 + w_3 Z^3 + \cdots,   \nonumber
\end{align}
with $Z = \sqrt{1-z/z_0}$ and where $f_1< 0$, $u_1<0$ and $w_1< 0$. 
Thus, it  follows that $M(z,0) = y(z) =  f(z)-w(z)u(z)$ has also the  
same kind of singularity: 
\begin{equation}\label{eqyrep}
y(z) = y_0 + y_1 Z + y_2Z^2 + y_3 Z^3 + \cdots.
\end{equation}

Hence, in order to complete the proof of Theorem~\ref{Th2} we have to show
the following properties:
\begin{enumerate}
\item If $Q_{\alpha_0u}\ne 0$ then the system (\ref{eqnewsystem.2}) is strongly connected.
\item We have $y_1 = 0$ in the expansion (\ref{eqyrep}).
\item If $F_0'(0) = 0$ and $Q_{\alpha_1}(F_0(0),0,0,0) \ne 0$ then $y_3 > 0$ in the expansion (\ref{eqyrep}).
\end{enumerate}
With these properties the singular structure of $M(z,0)$ at $z_0$ is precisely
of the form (\ref{eqTh2.1}). Furthermore, the asymptotics of $M_n$ follows in the following way.
Since the system (\ref{eqnewsystem.2}) is non-linear and strongly connected 
we know from \cite{0} that there exists
$a_1,a_2,a_3 \ge 0$ and $b\ge 1$ such that $f(z) = z^{a_1} F(z^b)$, $w(z) = z^{a_2} W(z^b)$, $u(z) = z^{a_3} U(z^b)$, 
where $F$, $W$, and $U$ have square-root singularities at $z= z_0^{1/b}$ but no other singularities on the circle
$|z|\le z_0^{1/b}$. It also follows that $M(z,0) = z^{a_1} F(z^b) - z^{a_2+a_3} W(z^b) U(z^b)$. If $a_1 \not\equiv a_2+a_3 \bmod b$
then $M(z,0)$ would have negative coefficients $M_n$ for $n\equiv a_2+a_3 \bmod b$ which is impossible. Thus,
$a_1 \equiv a_2+a_3 \bmod b$ and we have positive coefficients for $n\equiv a_1 \bmod b$ (if $n$ is sufficiently large) and
zero coefficients else. From $M(z,0) = z^{a_1} \widetilde M(z^b)$, where $\widetilde M(z)$ has $z=z_0^{1/b}$ as a singularity 
of type (\ref{eqTh2.1}) and no other singularities on the circle $|z|\le z_0^{1/b}$, we obtain
the asymptotics (\ref{eqTh2.2}).

Finally we comment on the computation of $z_0$. Let ${\bf J} = {\bf J}(f,w,u,z)$ denote the Jacobian matrix 
(with derivatics with respect to $f,w,u$) of 
the right hand side of (\ref{eqnewsystem.2}). Then we consider the extended system of equations
\begin{align}
f_0  &=  F_0(u_0 ) + z_0Q(f_0 , w_0 , z_0,u_0 ), \nonumber \\
u_0  &=  z_0 u_0   Q_{\alpha_0} (f_0 , w_0 , z_0,u_0 ) + z_0  Q_{\alpha_1} (f_0 , w_0 , z_0,u_0 ),  \label{eqnewsystem.3} \\
w_0   &= F_{0}'(u_0 ) + z_0 Q_u (f_0 , w_0 , z_0,u_0 ) + z_0 w_0    Q_{\alpha_0} (f_0 , w_0 , z_0,u_0 ).\nonumber \\
0 & = \det ({\bf I} - {\bf J}(f_0,w_0,u_0,z_0))  \nonumber
\end{align}
and search for the unique positive solution $(f_0,w_0,u_0,z_0)$ such that the 
spectral radius of ${\bf J}(f_0,w_0,u_0,z_0)$ equals $1$. This gives the correct value $z_0$.

\subsubsection{\textbf{Strong connectedness}}
Let $y_j = F_j(z,y_1,\ldots,y_d)$ a $d$-dimensional system of equations. The dependency di-graph of such a 
system consists of vertices $\{y_1,\ldots, y_d\}$ and there is an oriented edge from $y_i$ to $y_i$ if
$F_j$ depends on $y_i$, that is, $F_{j,y_i} \ne 0$. We say that the system is strongly connected if the
dependency di-graph is strongly connected (see \cite{0,Drm-randomtrees}). 
In our present situation our vertex set is $\{ f,u,w \}$. By assumption we have $Q_{\alpha_1}\ne 0$.
Thus, there is always an edge from $w$ to $f$. 

Suppose first the $Q_{\alpha_0\alpha_0}\ne 0$. 
Then by the second equation there is an edge from $f$ to $u$. By assumption we always have $Q_{\alpha_0u} \ne 0$
which implies that there is an edge from $u$ to $w$. This implies a circle $w \to f \to u \to w$ and consequently
stongly connectedness. 

Second suppose that $Q_{\alpha_0\alpha_1}\ne 0$ or $Q_{\alpha_1\alpha_1}\ne 0$. In this case
there is certainly an edge from $w$ to $u$. Furthermore, since $Q_{\alpha_0u} \ne 0$ there is an edge from $u$ to $w$ and
another one from $f$ to $w$. This again leads to a strongly connected di-graph and completes the proof of the first assertion.

\subsubsection{\textbf{The condition $y_1 = 0$}}
In order to prove that $y_1$ vanishes we recall first the approach by 
Bousquet-M\'elou and Jehanne \cite{BM-J}. Starting with the function $P(x_0,x_1,z,u)$ that is given
by (\ref{eqPrep}) we have to solve the system 
\begin{align}
P(f(z),y(z),z,u(z)) &= 0,  \nonumber \\
P_{x_0}(f(z),y(z),z,u(z)) &= 0, \label{eqoldsystem}  \\
P_u(f(z),y(z),z,u(z)) &= 0. \nonumber
\end{align}
Instead of searching for the functions
$f(z)$, $u(z)$ and $w(z)$ we equivalently search for 
$f(z)$, $y(z)$ and $u(z)$. It is also immediately clear that the singular condition
for the system (\ref{eqnewsystem}) implies that the system (\ref{eqoldsystem})
gets singular too. Consequently the functional determinant has to be zero,
evaluated at $(f(z_0), y(z_0), z_0, u(z_0))$. Since $P_{x_0} = P_u = 0$, we get 
\[
\det \left(  \begin{array}{ccc}
P_{x_0} & P_{x_1}  & P_u \\
P_{x_0x_0} & P_{x_0x_1}  & P_{x_0u} \\
P_{x_0u} & P_{x_1u}  & P_{uu}
\end{array} \right) = - P_{x_1} \left( P_{x_0x_0}P_{uu} - P_{x_0u}^2 \right) = 0. 
\]
Otherwise the implicit function theorem would imply that there is an analytic 
continuation. Since $P_{x_1} = -z Q_{\alpha_1}/u\ne 0$ (by assumption $Q_{\alpha_1}\ne 0$)
we obtain the relation $P_{x_0x_0}P_{uu} = P_{x_0u}^2$.

We now discuss the analytic function $P$ at the point  $(f_0,y_0,z_0,u_0) = (f(z_0), y(z_0), z_0, u(z_0))$
in more detail. We already know that $P_{x_0} = 0$. However, by differentiating \eqref{eqPrep} and using $F'_0(0)=0$  it
follows that 
\[
P_{x_0x_0} = z_0 Q_{\alpha_0\alpha_0} + 2 \frac {z_0}{u_0} Q_{\alpha_0\alpha_1} + \frac {z_0}{u_0^2} Q_{\alpha_1\alpha_1} > 0.
\]
Hence by the Weierstrass preparation theorem\footnote
{The Weierstrass preparation theorem says that every non-zero function
$F(z_1,\ldots,z_d)$ with $F(0,\ldots, 0 ) = 0$ 
that is analytic at $(0,\ldots,0)$ has a
unique factorisation $F(z_1,\ldots,z_d) = K(z_1,\ldots,z_d)
W(z_1;z_2,\ldots,z_d)$ into analytic factors, where $K(0,\ldots, 0)\ne 0$
and $W(z_1;z_2,\ldots,z_d) = z_1^d + z_1^{d-1} g_1(z_2,\ldots,z_d) + 
\cdots + g_d(z_2,\ldots,z_d)$ is a so-called Weierstrass polynomial,
that is, all $g_j$ are analytic and satisfy $g_j(0,\ldots,0) = 0$.}
 \cite{Kaup} it follows that $P$ can be locally written as
\begin{equation}\label{eqPKeq}
P(x_0,x_1,z,u) = K(x_0,x_1,z,u) \left( (x_0 - G(x_1,z,u))^2 - H(x_1,z,u) \right),
\end{equation}
where $K$, $G$ and $H$ are analytic function with the properties that
$K(f_0,y_0,z_0,u_0)\ne 0$, $G(y_0,z_0,u_0) = f_0$, and $H(y_0,z_0,u_0) = 0$.

Since $P = 0$ if and only if $(f-G)^2  =  H$ and 
\begin{align*}
P_{x_0} &= K_{x_0} \left( (f-G)^2 - H \right) + 2 K (f-G), \\
P_u &= K_u  \left( (f-G)^2 - H \right) + K \left( - 2(f-G)G_u - H_u \right)
\end{align*}
it follows from (\ref{eqoldsystem}) that 
\begin{equation}\label{eqHHu}
H(y(z),z,u(z)) = 0 \quad \mbox{and}\quad H_u(y(z),z,u(z)) = 0
\end{equation}
for $z$ close to $z_0$. We note that this is precisely a system of equations that appears
in the context of the quadratic method (see \cite{BM-J,DN11}).

Next we will show how the singular condition $P_{x_0x_0}P_{uu} = P_{x_0u}^2$ 
translates into $H_{uu}(y_0,z_0,u_0) = 0$. 
Since
\begin{align*}
P_{x_0x_0} &= K_{x_0x_0} \left( (f-G)^2 - H \right) 
+4K_{x_0}(f-G) + 2K,\\
P_{uu} &= K_{uu}\left( (f-G)^2 - H \right) 
+2K_{u}\left( - 2(f-G)G_u - H_u \right) \\ &+ 
K\left( 2G_u^2 - 2(f-G)G_{uu} - H_{uu} \right),\\
P_{x_0u} &= K_{x_0u} \left( (f-G)^2 - H \right) 
+2K_u(f-G) + K_{x_0} \left( - 2(f-G)G_u - H_u \right)\\
&+K\left(-2G_u - 2(f-G)G_{x_0u} \right)
\end{align*}
it follows that we have 
\begin{align*}
P_{x_0x_0} &= 2K,\\
P_{uu} &= (2 G_u^2 - H_{uu}) K,\\
P_{x_0u} &= -2G_u K
\end{align*}
for $(y,z,u) =(y_0,z_0,u_0)$. Consequently 
the condition $P_{x_0x_0}P_{uu} = P_{x_0u}^2$ implies 
$H_{uu}(y_0,z_0,u_0) = 0$.

In a similar (but much easier way) it also follows that
that $P_{x_1} = -K H_{x_1}$. This also implies
that $H_{x_1} \ne 0$ since $P_{x_1}\ne 0$ (by assumption $Q_{\alpha_1}\ne 0$).

Nest we recall that $u(z)$ and $y(z) = f(z) - u(z)w(z)$ have singular 
(and convergent) expansions of the
form 
\begin{align*}
u(z) &= u_0 + u_1 Z + u_2 Z^2 + u_3 Z^3 + \cdots, \\
y(z) &= y_0 + y_1 Z + y_2 Z^2 + y_3 Z^3 + \cdots, 
\end{align*}
where $Z = \sqrt{1 - z/z_0}$ and $u_1 < 0$. By using the Taylor expansion of
$H$ at $(y_0,z_0,u_0)$ and the property $H(y(z),z,u(z)) =0$ it follows that
\begin{align*}
0=& H_{x_1}\left( y_1 Z + y_2 Z^2 + y_3 Z^3 + \cdots \right) -z_0H_z Z^2 + \frac 12 H_{x_1x_1} 
\left( y_1^2 Z^2 + 2y_1y_2 Z^3 + \cdots \right) \\
& + H_{x_1u}\left( y_1u_1 Z^2 + (y_1u_2+y_2u_1) Z^3 + \cdots \right) - z_0H_{zu} (u_1 Z^3 + \cdots) \\
& - z_0 H_{x_1z} (y_1 Z^3 + \cdots) + \frac 16 H_{uuu}( u_1^3 Z^3 + \cdots)   + \frac 12 H_{x_1uu}( u_1^2 y_1 Z^3 + \cdots)\\
& + \frac 12 H_{x_1x_1u}( u_1 y_1^2 Z^3 + \cdots) + \frac 16 H_{x_1x_1x_1}( y_1^3 Z^3 + \cdots) + O(Z^4).
\end{align*}
By comparing coefficients of $Z$ this implies
\begin{align*}
0 =& H_{x_1} y_1, \\
0 =& H_{x_1} y_2 - z_0 H_z + \frac 12 H_{x_1x_1}y_1^2 + H_{x_1u}y_1u_1, \\
0 =&  H_{x_1} y_3 + H_{x_1x_1}y_1y_2 + H_{x_1u} (y_1u_2+y_2u_1) - z_0H_{zu}u_1- z_0 H_{yz} y_1 +\frac 16 H_{uuu}u_1^3\\
&+ \frac 12 H_{x_1uu} u_1^2 y_1+\frac 12 H_{x_1x_1u} u_1 y_1^2+\frac 16 H_{x_1x_1x_1} y_1^3.
\end{align*}
In particular, since $H_{x_1}\ne 0$ it follows that $y_1 = 0$, which completes the proof of the second property.

\subsubsection{\textbf{The condition $y_3> 0$}}
By taking also into account the second and third  relations above and using $y_1=0$  we get
\begin{align*}
y_2 &= \frac{z_0H_z}{H_{x_1}},\\
y_3 &= \frac{u_1}{H_{x_1}} 
\left( z_0H_{zu} - H_{x_1u} y_2 - \frac 16 H_{uuu}u_1^2 \right).
\end{align*}
By doing the same procedure as above for $H_u$ (where we can only use that $H_{uu} = 0$) we also
get the relation 
\begin{align}\label{u1}
H_{x_1u} y_2 - z_0H_{zu} + \frac 12 H_{uuu}u_1^2 = 0.
\end{align}
Hence $y_3$ can be also represented as
\[
y_3 = \frac{2u_1z_0}{3H_{x_1}^2} 
\left( H_{x_1}H_{zu} -  H_zH_{x_1u}\right).
\]
We already know that $u_1 < 0$ and $H_{x_1} \ne 0$. Thus it remains to show that
\begin{equation}\label{eqy3remains}
 H_{x_1}H_{zu} -  H_zH_{x_1u} \ne 0,
\end{equation}
where we evaluate at $(y_0,z_0,u_0)$. This will show that $y_3 \ne 0$. 

By slightly more involved computations as above it follows that (\ref{eqy3remains}) holds if and only if
the functional determinant
\begin{align}
\Delta = \left| \begin{array}{ccc} 
P_{x_0} & P_{x_1} & P_{z}  \\
P_{x_0x_0} & P_{x_0x_1} & P_{x_0z}  \\
P_{x_0u} & P_{x_1u} & P_{uz}  \\
\end{array}\right| = &- P_{x_1} \left( P_{x_0x_0}P_{uz}-P_{x_0z} P_{x_0u} \right)  \nonumber \\
&+ P_z \left( P_{x_0x_0}P_{ux_1}-P_{x_0x_1} P_{ux_0}  \right)  \ne 0,   \label{eqfuncdet} 
\end{align}
where we evaluate at $(f_0,y_0,z_0,u_0)$, so that $P_{x_0} = 0$.

Let us assume for a moment that (\ref{eqfuncdet}) is satisfied. Then by the 
implicit function theorem it follows that the system 
\begin{equation}\label{eqPPP}
P = 0, \quad P_{x_0} = 0, \quad P_u = 0
\end{equation}
has a unique solution $x_0 = \widetilde f(u)$, $x_1 = \widetilde y(u)$, $z = \widetilde z(u)$ 
with $\widetilde f(u_0) = y_0$, $\widetilde y(u_0) = y_0$, $\widetilde z(u_0) = z_0$.
Actually the converse is {\it almost true} and this is the strategy of our proof.

First we observe that assuming (\ref{eqfuncdet})  the system (\ref{eqPPP}) has a unique solution $x_0 = \widetilde f(u)$, $x_1 = \widetilde y(u)$, $z = \widetilde z(u)$ 
with $\widetilde f(u_0) = y_0$, $\widetilde y(u_0) = y_0$, $\widetilde z(u_0) = z_0$.
For this purpose we go back to the system (\ref{eqnewsystem.2}), which  is equivalent
(after the substitution $w = (f-y)/u$). The essential difference is that the {\it free variable} of 
(\ref{eqnewsystem.2}) is $z$ and in (\ref{eqPPP}) it is $u$. From the theory of positive strongly connected
polynomial systems it follows that at the critical point $z = z_0$ there are precisely two local solutions, namely
\begin{align}
u_+(z) &= u_0 + u_1 Z + u_2 Z^2 + u_3 Z^3 + \cdots \quad\mbox{and} \nonumber \\
u_-(z) &= u_0 - u_1 Z + u_2 Z^2 - u_3 Z^3 + \cdots  \label{equsol}
\end{align}
with $Z = \sqrt{1 - z/z_0}$, and similarly for $f(z)$ and $w(z)$. 
(The method reduces the system first to a single non-linear positive equation for which this
can be easily observed; see \cite{0,Drm-randomtrees}.) 
It is immediate that the two solutions of (\ref{equsol}) correspond to a single
function $z=\widetilde z(u)$ with $\widetilde z'(u_0) = 0$ with $z = \widetilde z(u_+(z)) = \widetilde z(u_-(z))$, and it is 
 clear that $\widetilde f(u)$ and $\widetilde y(u)$ are unique. 
This is a strong indication that (\ref{eqfuncdet}) should be satisfied.
However, there might be exceptional situations as we argue next.

We start by solving the  simpler system
\[
P = 0, \quad P_{x_0} = 0
\]
with solutions $x_0 = \overline f(z,u)$,  $x_1 = \overline y(z,u)$
with $\overline f(z_0,u_0) = f_0$, $\overline y(z_0,u_0) = y_0$. This is certainly 
possible since the functional determinant $P_{x_0}P_{x_0x_1} - P_{x_1}P_{x_0x_0} =  - P_{x_1}P_{x_0x_0}\ne 0$
as long as $x_0,x_1,z, u$ are positive.

Now  we substitute these solution functions into the equation $P_u = 0$ 
and obtain 
\[
P_u ( \overline f(z,u), \overline y(z,u), z, u) = \overline P(z,u) = 0,
\]
that has the unique solution $z = \widetilde z (u)$ with $\widetilde z(u_0) = z_0$. 
Note that 
\begin{align*}
\overline P_z &= P_{ux_0} \overline f_z + P_{ux_1} \overline y_z + P_{uz} \\
&= P_{ux_0} \frac{ P_{x_0x_1}P_z - P_{x_0z} P_{x_1}  }{P_{x_1}P_{x_0x_0}}  - P_{ux_1}  \frac{P_z}{P_{x_1}} + P_{uz} \\
&= -\frac{\Delta}{P_{x_1}P_{x_0x_0}} 
\end{align*}
so that $\overline P_z \ne 0$ if and only if (\ref{eqfuncdet}) holds. 

Thus, it suffices to consider an equation of the form $\overline P(z,u) = 0$ for which we know that
there is a unique solution $z = \widetilde z(u)$ with $\widetilde z(u_0) = z_0$, for which we now assume
that $\overline P_z(z_0,u_0) = 0$. Since $\overline P$ is non-zero there exists $r\ge 1$ such
that $\overline P_z(z_0,u_0) =\overline P_{z^2}(z_0,u_0) = \cdots = \overline P_{z^r}(z_0,u_0) = 0$,
but  $\overline P_{z^{r+1}}(z_0,u_0) \ne 0$. By the Weierstrass preparation theorem there exists an
analytic function $\overline K(z,u)$ with $\overline K(z_0,u_0)\ne 0$ and analytic functions 
$c_0(u), \ldots, c_r(u)$ with $c_j(u_0) = 0$, $0\le j\le r$, such that
\begin{equation}\label{eqPrep2}
\overline P(z,u) = \overline K(z,u) \left( (z-z_0)^{r+1} + c_r(u) (z-z_0)^{r} + \cdots + c_0(u) \right).
\end{equation}
Since we know that there exists a unique solution  $z = \widetilde z(u)$ with $\widetilde z(u_0) = z_0$
of the equation $\overline P(z,u) = 0$, it follows that the polynomial is an $(r+1)$-th power:
\[
(z-z_0)^{r+1} + c_r(u) (z-z_0)^{r} + \cdots + c_0(u) = (z-\widetilde z(u))^{r+1}.
\]
Next we note that the function $u(z)$ is strictly increasing for $0\le z \le z_0$. Thus the inverse function
$\widetilde z(u)$ exists for $0\le u \le u_0$ and can be analytically continued to a region $G$ that 
covers the real interval $[0,u_0]$. (Note that by assumption $u'(0) = Q_{\alpha_1}(F_0(0),0,0,0) \ne 0$ 
so that the inverse function $\widetilde z(u)$ is analytic at $u=0$, too.)
Since 
\begin{equation}\label{eqPrep2}
\overline P(z,u) = \overline K(z,u) (z-\widetilde z(u))^{r+1} .
\end{equation}
holds in a neighborhood of $(z_0,u_0)$ it follows that
$\overline P_z(\widetilde z(u),u) =0$ holds in a neighborhood of $u_0$
and, thus, for all $u\in G$.

Summing up, if (\ref{eqfuncdet}) does not hold at $(y_0,z_0,u_0)$, that is,
$\Delta = 0$ at $(f_0,y_0,z_0,u_0)$ then $\Delta = 0$ 
evaluated at $(\widetilde f(u), \widetilde y(u), \widetilde z(u), u)$ for all $u \in [0,u_0]$.
  
We can now finalize the proof by showing that $\Delta \ne 0$
for $u$ sufficiently close to $0$ provided that $F_0'(0) = 0$. 
Note first that the condition $F_0'(0) = 0$ implies $w(0) = 0$,
where $w = (\widetilde f(u)-\widetilde y(u))/u$.
By using (\ref{eqPrep}) it follows that 
\begin{align*}
u^4 \Delta &= \widetilde z(u) Q_{\alpha_1} \left(  Q_{\alpha_1} + w Q_{\alpha_1\alpha_1} - 2 w Q_{\alpha_0\alpha_1} \right) + O(u^2)\\
& = \widetilde z'(0) u Q_{\alpha_1}(F_0(0),0,0,0) ^2 + O(u^2) \\
&= Q_{\alpha_1}(F_0(0),0,0,0) u + O(u^2).
\end{align*}
By assumption $Q_{\alpha_1}(F_0(0),0,0,0)\ne 0$. Hence $\Delta$ is not identically 0 , so that  $\Delta \ne 0$ for $u$ sufficiently close to $0$.
As argued above this also implies that (\ref{eqfuncdet}) holds, and consequently $y_3 \ne 0$.
Since $y(z) = M(z,0)$ has non-negative coefficients this implies $y_3> 0$.

\section{Central Limit Theorems for Additional Parameters}\label{sec:clt}

Let $M(z,x,u)$ denote the generating function of rooted planar maps, where the variable $z$ 
corresponds to the number of edges, $x$ to the number of vertices and $u$ to the root face valency.
Then by the usual combinatorial decomposition of maps we have
\[
M(z,x,u) = x + zu^2 M(z,x,u)^2 + zu \frac{M(z,x,1)- u M(z,x,u)}{1-u}.
\]
Thus, for every positive $x$ this is a catalytic equation of the form (\ref{eqBMJ})
so that Proposition~\ref{ProTh2} and Theorem~\ref{Th2} apply. In particular we obtain an expansion
and asympotics of the form
\begin{equation}\label{eqsingexpx}
M(z,x,1) = a_0(x) + a_2(x)\left( 1- \frac z{\rho(x)} \right) + a_3(x)\left( 1- \frac z{\rho(x)} \right)^{3/2} + \cdots,
\end{equation}
where $z=\rho(x)$ satisfies the equation
\begin{multline*}
768\,{x}^{4}{z}^{4}-1536\,{x}^{3}{z}^{4}-512\,{x}^{3}{z}^{3}+2304\,{x}
^{2}{z}^{4}+768\,{x}^{2}{z}^{3}-1536\,x{z}^{4}\\ +96\,{x}^{2}{z}^{2} +768
\,x{z}^{3}+768\,{z}^{4}-96\,x{z}^{2}-512\,{z}^{3}+96\,{z}^{2}-1 = 0
\end{multline*}
with $\rho(1) = \frac 1{12}$ and where
 $a_0(1) = \frac 43$, $a_2(1) = -\frac 43$, $a_3(1) = \frac 83$,
and consequently
\[
[z^n]\, M(z,x,1)  = c(x) n^{-5/2} \rho(x)^{-n} \left( 1 + O\left( \frac 1n \right) \right).
\]
Actually all the  functions $\rho(x)$, $c(x)$, and $a_j(x)$ are not only defined for positive $x$ but extend
to analytic functions around the positive real axis, and by inspection of the proof even the 
asymptotics can be extended to non-real $x$ that are close to the positive real axis.

Let $X_n$ denote the random variable equal to  the number of vertices in a random planar rooted map
with $n$ edges, where each map of size $n$ is considered to be equally likely. Then the probability
generating function $\mathbb{E}[x^{X_n}]$ can be written as
\[
\mathbb{E}[x^{X_n}] = \frac{ [z^n]\, M(z,x,1) }{ [z^n]\, M(z,1,1)} = \frac{c(x)}{c} \left( \frac{\rho(1)}{\rho(x)} \right)^n
\left( 1 + O\left( \frac 1n \right) \right).
\]
At this stage we can apply standard tools (see \cite[Chapter 2]{Drm-randomtrees}) 
to obtain a central limit theorem for $X_n$ of the form
$(X_n - \mu n)/{\sqrt{\sigma^2 n}} \to \mathcal{N}(0, 1)$, where
$\mu = - {\rho'(1)}/{\rho(1)}$ and $\sigma^2 = \mu + \mu^2 - {\rho''(1)}/{\rho(1)}$.
Since $\rho'(1) = - \frac 1{24}$ and $\rho''(1) =  \frac {19}{384}$ we immediately obtain
$\mu = \frac 12$ and $\sigma^2 = \frac {5}{32}$. We also have 
$\mathbb{E}[X_n]= \mu n + O(1)$ and $\mathbb{V}{\rm ar}[X_n]  = \sigma^2 n +  O(1)$.
In this special case Euler's relation and duality can be used to obtain (the even more precise representation)
$\mathbb{E}[X_n] = n/2 + 1$.\footnote{This central limit theorem seems to be a folklore result. However, to the
best of our knowledge it was first explicitly mentioned by the second author at the Alea-meeting 2010 in Luminy: 
{{\tt https://www-apr.lip6.fr/alea2010/}} .}

\medskip

Actually we can easily generalize Proposition~\ref{ProTh2} and Theorem~\ref{Th2} in order to 
obtain the following central limit theorem.

\begin{theo}\label{Th2-ext}
Suppose that $Q$ is a polynomial in $\alpha_0,\alpha_1,z,x,u$ with non-negative coefficients
that depends (at least) on $\alpha_1$, that is, $Q_{\alpha_1} \ne 0$, and
$F_0(x,u)$ is another polynomial with non-negative coefficients.
Let $M(z,x,u)$ be the power series solution of the equation
\begin{equation}\label{eqBMJx}
M(z,x,u) = F_0(x,u) + z Q\left( M(z,x,u), \frac{M(z,x,u)-M(z,x,0)}u, z,x,u \right).
\end{equation}
Furthermore  assume that $Q$ is not linear in $\alpha_0$ and $\alpha_1$, that is,
$Q_{\alpha_0\alpha_0} \ne 0$, or $Q_{\alpha_0\alpha_1} \ne 0$ or $Q_{\alpha_1\alpha_1} \ne 0$.
Additionally  assume that $Q_{\alpha_0u} \ne 0$, $\frac{\partial F_0}{\partial u}(x,0) = 0$, $Q_{\alpha_1}(0,0,1,0,0) \ne 0$
and that (according to Theorem~\ref{Th2})
$[z^n]\, M(z,1,0) > 0$ for $n \equiv a \bmod b$, whereas $[z^n]\, M(z,1,0) = 0$
for $n \not\equiv a \bmod b$.

Let $X_n$ with $n\equiv a \bmod b$ be a sequence of random variables defined by
\[
\mathbb{E}[x^{X_n}] = \frac{ [z^n]\, M(z,x,1) }{ [z^n]\, M(z,1,1)}.
\]
For positive $x$, let $\rho(x)>0$ denote the radius of convergence of $z\mapsto M(z,x,0)$.
Then $\rho(x)$ can be extended to an analytic function around the positive real axis and we have with
\[
\mu = - \frac{\rho'(1)}{\rho(1)}, \quad \sigma^2 = \mu + \mu^2 - \frac{\rho''(1)}{\rho(1)}
\]
the following asymptotic moment properties:
\[
\mathbb{E}[X_n]= \mu n + O(1) \quad\mbox{and}\quad \mathbb{V}{\rm ar}[X_n]  = \sigma^2 n +  O(1),
\]
for $n\equiv a \bmod b$. 
Furthermore, if $\sigma^2 \ne 0$ then we also have a central limit theorem of the form
\[
\frac{X_n - \mathbb{E}[X_n]}{\sqrt{ \mathbb{V}{\rm ar}[X_n] }} \to \mathcal{N}(0,1) \qquad (n\equiv a \bmod b).
\]
\end{theo}

\begin{proof}
It is easy to show that, for every positive $x$, we can apply Proposition~\ref{ProTh2} and Theorem~\ref{Th2}
and obtain (for $n\equiv a \bmod b$)
\begin{equation}\label{eqMzx1coeff}
[z^n]\, M(z,x,1)  = c(x)\, n^{-5/2} \rho(x)^{-n} \left( 1 + O\left( \frac 1n \right) \right)
\end{equation}
for some positive valued function $c(x)$. Note that the error term comes from the
remainder terms 
\[
a_4(x) \left( 1- \frac z{\rho(x)} \right)^2 + O\left(\left( 1- \frac z{\rho(x)} \right)^{5/2} \right)
\]
in the singular expansion (\ref{eqsingexpx}) of $M(z,x,1)$. The term 
$a_4(x) \left( 1- z/{\rho(x)} \right)^2$ has no asymptotic contribution,
wheres the other term gives rise to the error term $O(n^{-7/2} \rho(x)^{-n})$; see also 
\cite{Drm-randomtrees}. This proves (\ref{eqMzx1coeff}). Furthermore, since
$F_0$ and $Q$ are polynomials it follows that $\rho(x)$ is an algebraic function since
it is determined by the algebraic system of equations (\ref{eqnewsystem.3}), where we just
have to add the algebraic dependence on $x$. 

Actually it can be shown that $\rho(x)$ has no
singular point for $x> 0$. As explained in the proof of Proposition~\ref{ProTh2}, we can reduce
the solution of the catalytic equation to a system of three positive polynomial equations.
Such a system can be reduced to a single equation 
$u(x,z) = F(x,z,u(x,u))$ in one unknown function $u=u(x,z)$, where $F = F(x,z,u)$ has a power series
expansion with non-negative coefficients (see \cite{Drm-randomtrees}). Note that we certainly have 
$F_z \ne  0$ and $F_{uu} \ne 0$.
The system of equations that determines the values $z = \rho(x)$ and $u = u(x,\rho(x))$, where the solution 
function $z\mapsto u(x,u)$ gets singular, is given
by 
\[
u = F(x,z,u), \quad   1 = F_u(x,z,u).
\]
The functional determinant of this system, when we solve it for $z = \rho(x)$ and $u = u(x,\rho(x))$, 
is given by
\[
F_z F_{uu} - (F_u-1) F_{uz} = F_z F_{uu} \ne 0.
\]
By the implicit function theorem  $z = \rho(x)$, as well as  $u = u(x,\rho(x))$  are analytic.
Moreover $\rho'(x) = - F_x/F_z < 0$, since $F_x > 0$ and $F_z > 0$. 

By the methods of \cite{Drm-randomtrees} it also follows that the singular expansion  
(\ref{eqsingrep-3}), where $Z$ has to replaced by $\sqrt{1 - z/\rho(x)}$ and all coefficient
functions $f_j$, $u_j$, $w_j$ depend on $x$, can be extended to complex $x$ that are sufficiently
close to the positive real axis. Accordingly the asymptotic expansion (\ref{eqMzx1coeff}) holds
uniformly if $x$ varies in a compact subset of the complex plane, where $\rho(x)$ is well defined.

As mentioned above this is sufficient to prove the asymptotic expansion for 
$\mathbb{E}[X_n]$, $\mathbb{V}{\rm ar}[X_n]$, as well the central limit theorem.
\end{proof}

We note the crucial point in the proof of Theorem~\ref{Th2-ext} was to prove 
a singular expansion of the form (\ref{eqsingexpx}) that holds in a complex
neighborhood of $x=1$. We finally add a theorem for catalytic equations, where 
we do not necessarily have a polynomial equation with non-negative 
coefficients. Again, we assume that there is an additional variable $x$, 
where we non necessarily assume that the defining catalytic equations 
contains only non-negative coefficients. This kind of approach was first applied
in \cite{DP13}, where the number of faces of given valency in random planar maps
was discussed; see below. (It was first stated without a proof in \cite{DYu}).  

\begin{theo}\label{Th2-ext-2}
Suppose that $M(z,x,u)$ and $M_1(z,x)$ are the solutions of 
the catalytic equation $P(M(z,x,u),M_1(z,x),z,x,u) = 0$, 
where the function $P(x_0,x_1,z,x,u)$ is analytic and $M_1(z,{1})$ has
a singularity at $z= z_0$ of form 
\begin{equation}\label{eqM1sing}
M_1(z,1) = y_0 + y_2\left( 1 - \frac z{z_0} \right)+ y_3\left( 1 - \frac z{z_0} \right)^{3/2} + \cdots,
\end{equation}
with $y_3\ne 0$ 
such that for  $x_0 = M(z_0,1,u_0)$, $x_1 = M_1(z_0,{1})$, $z = z_0$, $x=1$, and $u=u_0$
we have 
\begin{equation}\label{eqTh2-ext-2}
P = 0, \quad P_u = 0, \quad P_{x_0} = 0, \quad P_{x_1} \ne 0, \quad 
P_{x_0x_0} P_{uu} = P_{x_0u}^2.
\end{equation}
Furthermore, let $z = \rho(x)$, $u = u_0(x)$, 
$x_0 = x_0(x)$, $x_1 = x_1(x)$ 
for ${x}$ close to ${1}$ 
be defined by $\rho({1}) = z_0$, $u_0({1}) = u_0$, 
$x_0({1}) = M(z_0,1,u_0)$, $x_1({1})= M_1(z_0,{ 1})$ and
by the system
\begin{equation}\label{eqTh2-ext-2-2}
P = 0, \quad P_u = 0, \quad P_{x_0} = 0, \quad 
P_{x_0x_0} P_{uu} = P_{x_0u}^2.
\end{equation}

Then for ${x}$ close to ${1}$ the function $M_1(z,x)$
has a local singular representation of the form
\begin{equation}\label{eqM1sing2}
M_1(z,x) = a_0(x) + a_2(x)\left( 1 - \frac z{\rho(x)} \right)+ a_3(x)\left( 1 - \frac z{\rho(x)} \right)^{3/2} + \cdots
\end{equation}
where the functions $a_j(x)$ are analytic at $x=1$ and satisfy $a_j(1) = y_j$.
\end{theo}

\begin{proof}
As in the proof of Theorem~\ref{Th2} we can replace the (catalytic) equation
$P(M(z,x,u),M_1(z,x),z,x,u) = 0$ by 
\[
\left( M(z,x,u)  - G(M_1(z,x),z,x,u)  \right)^2 = H(M_1(z,x),z,x,u)
\]
around $z=z_0$, $x=1$, $u=u_0$. In particular we have
\[
H = 0, \quad H_u = 0, \quad H_{uu} = 0, \quad H_{x_1} \ne 0
\]
for $x_0 = M(z_0,1,u_0)$, $x_1 = M_1(z_0,{1})$, $z = z_0$, $x=1$, and $u=u_0$.

In the next step we set $x=1$ and apply the methods from \cite[Lemma 2]{DN11} that ensure
that there exist precisely two (local) solutions $u(z)$ and $y(z)$ of the system of equations 
\[
H(y(z),z,1,u(z)), \quad H_u(y(z),z,1,u(z))
\]
with $y(z_0) = M_1(z_0,{1})$ and $u(z_0) = u_0$ and with local expansions 
\[
u(z) = u_0 \pm u_1 Z + u_2 Z^2 \pm u_3 Z^3 + \cdots, \qquad
y(z) = y_0 + y_2 Z^2 \pm y_3 Z^3 + \cdots,
\]
where $Z = \sqrt{1-z/z_0}$ (and the signs are either all positive or all negative).
By assumption, one of these two solutions has to be equal to
$M_1(z,1)$ which implies that $y_3 \ne 0$.  

By the methods of the proof of Theorem~\ref{Th2} we also have
\[
y_3 = \frac{2u_1z_0}{3H_{x_1}^2} 
\left( H_{x_1}H_{zu} -  H_zH_{x_1u}\right).
\]
Recall that $H_{x_1} \ne 0$. Thus it follows that $z_0\ne 0$, $u_1\ne 0$, and
$H_{x_1}H_{zu} -  H_zH_{x_1u} \ne 0$. Furthermore, since $y_2 = z_0 H_z/H_{x_1}$ 
it follows from the relation (\ref{u1}) that
\[
H_{uuu} = \frac{2 z_0}{u_1^2} \left( H_{x_1}H_{zu} -  H_zH_{x_1u}\right) \ne 0.
\]

We want to do a similar analysis for $x$ close to $1$. For this purpose we
have to check whether the conditions (\ref{eqTh2-ext-2}) can be extended to 
$x$ different from $1$, that is, whether is is possible to solve the system
(\ref{eqTh2-ext-2-2}) for $z = \rho(x)$, $u = u_0(x)$, 
$x_0 = x_0(x)$, $x_1 = x_1(x)$ (if $x$ is close to $1$). Note that the 
condition $P_{x_1} \ne 0$ certainly extends to a neighborhood.
By the same procedure as in the proof of Theorem~\ref{Th2} it follows that
the system (\ref{eqTh2-ext-2-2}) is equivalent to the system
\begin{equation}\label{eqHHH}
H = 0, \quad H_u = 0, \quad H_{uu} = 0
\end{equation}
for $x_1 = x_1(x)$, $z = \rho(x)$, $u = u_0(x)$, Note that $x_1(x)$, $\rho(x)$, $u_0(x)$
are the same functions as above; and the function $x_0(x)$ can be recovered by
$x_0(x) = M(\rho(x),x,u_0(x))$. Now the functional determinant of the system
(\ref{eqHHH}) is given by
\[
\left|    
\begin{array}{ccc}
H_{x_1} & H_{ux_1} & H_{uux_1} \\
H_{z} & H_{uz} & H_{uuz} \\
H_u & H_{uu} & H_{uuu} 
\end{array}
\right| = H_{uuu} \left( H_{x_1}H_{uz}-H_z H_{ux_1} \right)
\]
which is non-zero at $x_1 = M_1(z_0,{1})$, $z = z_0$, $x=1$, and $u=u_0$.
Hence by the implicit function theorem the system (\ref{eqHHH}) has an
analytic (and unique) local solution $x_1 = x_1(x)$, $z = \rho(x)$, $u = u_0(x)$
with $x_1(1) = M_1(z_0,{1})$, $\rho(1) = z_0$, $u_0(1) = u_0$.

Summing up, we can apply the same techniques as in \cite[Lemma 2]{DN11} 
that are now valid uniformly in a small (complex) neighborhood of $x=1$ 
and leads to an expansion of the form (\ref{eqM1sing2}). 
\end{proof}

Expansions of the form (\ref{eqM1sing}) or (\ref{eqM1sing2}), respectively, are 
in particular useful if $z=z_0$ or $z=\rho(x)$ is the only singularity on the
slit disc 
\[
\{ z \in \mathbb{C}: |z|< z_0 + \varepsilon \} \setminus [z_0,\infty) \quad \mbox{or}\quad
\{ z \in \mathbb{C}: |z|< |\rho(x)| + \varepsilon \} \setminus [z_0,\infty)
\]
for some $\varepsilon > 0$. In this case it follows directly that, as $n\to\infty$,
\[
[z^n]\, M_1(z,1) = \frac{3 y_3}{4 \sqrt\pi} z_0^{-n} n^{-5/2} \left( 1 + O\left( \frac 1n \right) \right)
\]
or 
\[
[z^n]\, M_1(z,x) = \frac{3 a_3(x)}{4 \sqrt\pi} \rho(x)^{-n} n^{-5/2} \left( 1 + O\left( \frac 1n \right) \right).
\]
Thus, if 
\[
M_1(z,x) = \sum_{n\ge 0} \mathbb{E}[x^{X_n}]\cdot [z^n]\, M_1(z,1) \cdot z^n
\]
encodes the distribution of a sequence of random variables $X_n$ it follows that
\[
\mathbb{E}[x^{X_n}] = \frac{ [z^n]\, M_1(z,x)}{ [z^n]\, M_1(z,1) } = \frac{a_3(x)}{y_3} \left( \frac {z_0}{\rho(x)} \right)^n
\left( 1 + O\left( \frac 1n \right) \right).
\]
and we obtain a central limit theorem by standard tools (see \cite{Drm-randomtrees} and the above discussion).

\begin{ex}
Let $k \ge 2$ be a fixed integer and let $M(z, x, u)$ be the ordinary
generating function enumerating rooted planar maps, where $z$ corresponds to the 
number of edges, $x$ to the number of non-root faces of degree $k$, 
and $u$ to the root-face degree. In \cite[Lemma 2]{DP13} is was shown that 
$M(z,x,u)$ satisfies the equation
\begin{align*}
M(z,x,u) & \left( 1- z(u-1)u^{-k+2} \right) 
= 1 + zu^2  M(z,x,u) \\
&+ zu \frac{uM(z,x,u) - M(z,x,1)}{u-1} \\
& - z(x-1)u^{-k+2} G(z,x,M(z,x,1),u),
\end{align*}
where $G(z,x,y,u)$ is a polynomial of degree $k-2$ in $u$ with coefficients that
are analytic functions in $(z, x, y)$ for $|z| \le 1/10$,  $|x - 1| \le 2^{1-k}$,
and $|y| \le 2$. It should be noted that the function $G$ is not explicitly given
but is (one of) the solution(s) of in infinite system of equations that can be
solved with the help of Banach's fixed point theorem.

Clearly, $M(z,1,u)$ is just the usual planar map counting generating function for which 
we know that $M(z,1,1)$ is explicitly given by (\ref{eqMz1}) so that all assumptions
of Theorem~\ref{Th2-ext-2} are satisfied. Alternatively we could have used
Theorem~\ref{Th2} to obtian the local expansion of $M(z,1,1)$. 
Furthermore a central limit theorem follows, 
where $X_n$ is just the number of non-root faces of valency $k$ in a random
planar map with $n$ edges. (This is also one of the main results of \cite{DP13}.)
\end{ex}

\begin{ex}
We say that a face is a pure $k$-gon ($k \ge 2$) if it is incident exactly to $k$
different edges and $k$ different vertices. 
let $P(z, x, u)$ be the ordinary
generating function enumerating rooted planar maps, where $z$ corresponds to the 
number of edges, $x$ to the number of non-root faces that are pute $k$-gons, 
and $u$ to the root-face degree. Similarly to the previous case it can be shown (see \cite{Yu}) that
$P(z,x,u)$ satisfies an equation of the form
\begin{align*}
P(z,x,u) & = 1 + zu^2  P(z,x,u) + zu \frac{uP(z,x,u) - P(z,x,1)}{u-1} \\
& - z(x-1)u^{-k+2} \tilde G(z,x,P(z,x,1),u),
\end{align*}
where $\tilde G(z,x,y,u)$ is a polynomial of degree $k-2$ in $u$ with coefficients that
are analytic functions in $(z, x, y)$ for $|z| \le 1/10$,  $|x - 1| \le 2^{1-k}$,
and $|y| \le 2$.

Again, it we set $x=1$ we recover $M(z,u) = P(z,1,u)$ so that all assumptions of
Theorem~\ref{Th2-ext-2} are satisfied. Hence, for fixed $k\ge 2$, the number of
pure $k$-gons in a random planar map satisfies a central limit theorem.
\end{ex}

\begin{ex}
A planar map is simple if is has no loops and no multiple edges. 
The corresponding generating function $S(z,u)$ (where $z$ corresponds to the number of edges
and $u$ to the root face valency) satisfies the catalytic equation (see \cite{Yu})
\begin{align*}
S(z,u) &= 1 + zu^2 S(z,u)^2 + zu \frac{uS(z,u) - S(z,1)}{u-1} \\
&- zuS(z,u)S(z,1) - (S(z,u)-1)(S(z,1)-1)
\end{align*}
and the solution $S(z,1)$ is explicitly given by
\[
S(z,1) = \frac{1+ 20 z - 8 z^2 + (1-8z)^{3/2}}{2(z+1)^3}.
\]
Let $k \ge 2$ be a fixed integer and let $S(z, x, u)$ be the ordinary
generating function enumerating simple rooted planar maps, where $z$ corresponds to the 
number of edges, $x$ to the number of non-root faces of degree $k$, 
and $u$ to the root-face degree. In \cite{Yu} is was shown that 
$S(z,x,u)$ satisfies the equation
\begin{align*}
S(z,x,u) &  
= 1 + zu^2  S(z,x,u) + zu \frac{uS(z,x,u) - S(z,x,1)}{u-1} \\
&- zuS(z,x,u)S(z,x,1) - (S(z,x,u)-1)(S(z,x,1)-1) \\
& +(x-1)\Biggl( zu^{-k+2} S(z,x,u)G_1(z,x,S(z,x,1),u) \\
& \qquad \qquad \qquad  - zu S(z,x,u)G_2(z,x,S(z,x,1)) \\
& \qquad \qquad \qquad  - (S(z,x,u)-1)G_3(z,x,S(z,x,1)) \Biggr), 
\end{align*}
where $G_1(z,x,y,u)$ is a polynomial of degree $k-2$ in $u$ with coefficients that
are analytic functions in $(z, x, y)$ for $|z| \le 2/25$,  $|x - 1| \le 2^{-k-5}$,
and $|y-1| \le 2/5$. Similarly the functions $G_2(z,x,y)$ and $G_3(z,x,y)$ are
analytic functions in $(z, x, y)$ for $|z| \le 2/25$,  $|x - 1| \le 2^{-k-5}$,
and $|y-1| \le 2/5$.

Again all assumptions of Theorem~\ref{Th2-ext-2} are satisfied. Hence, for fixed $k\ge 2$, the number of
faces of valency $k$ in a random simple planar map satisfies a central limit theorem.
\end{ex}

\end{document}